\documentclass[12pt,reqno]{amsart}
\usepackage{graphicx}
\usepackage{psfrag}
\usepackage{pxfonts}
\usepackage{mathrsfs}
\usepackage{color}
\usepackage{amsmath,amssymb}
\numberwithin{equation}{section}

\def \torus {{\mathbb{T}}}
\def \real {{\mathbb{R}}}
\theoremstyle{plain}
\newtheorem{maintheorem}{Theorem}

\newcommand{\R}{\mathbb{R}}

\newcommand{\Z}{\mathbb{Z}}

\newtheorem{theorem}{Theorem}[section]

\newtheorem{proposition}[theorem]{Proposition}
\newtheorem{lemma}[theorem]{Lemma}
\newtheorem{definition}[theorem]{Definition}

\theoremstyle{remark}
\newtheorem{remark}[theorem]{Remark}

\begin{document}
\title[Lyapunov exponents everywhere and rigidity]
{Lyapunov exponents everywhere and rigidity}
%{Rigidity of systems all of whose orbits have Lyapunov exponents}
\author[F. Micena]{Fernando Pereira Micena}
\address{Instituto de Matem\'{a}tica e Computa\c{c}\~{a}o,
  IMC-UNIFEI, Itajub\'{a}-MG, Brazil.}
\email{fpmicena82@unifei.edu.br}

\author[R. de la Llave]{Rafael de la Llave}
\address{
School of Mathematics,
Georgia Institute of Technology,
686 Cherry St. ,  Atlanta GA 30332-1160 }
\email{rafael.delallave@math.gatech.edu}

\thanks{R.L. was partially supported by NSF grant DMS-1800241. F. Micena appreciates the unconditional support of his family. Also, Micena is grateful to Rafael de la Llave for the opportunity to write an article with him. The authors thank the anonymous referees for their suggestions and valuable comments.}
%%%%%%%%%%%%%%%%%%%%%%%%%
\baselineskip=18pt              %% DRAFT MODE -- double spaced.

%%%%%%%%%%%%%%%%%%%%%%%%%
\begin{abstract}
In the present work we obtain rigidity results analysing the set of regular points, in the sense of Oseledec's Theorem. It is presented a study on the possibility of  an Anosov diffeomorphisms having all Lyapunov
exponents defined everywhere. We prove that this condition implies local rigidity of an Anosov automorphism of the torus $\mathbb{T}^d, d \geq 3,$ $C^1-$close to a linear automorphism diagonalizable over $\mathbb{R}$ and such that its characteristic polynomial is irreducible over $\mathbb{Q}.$

\end{abstract}

\subjclass[2010]{}
\keywords{}

\maketitle

%\tableofcontents

\section{Introduction and statements of results}\label{sec:intro}
One of the most celebrated theorems in modern dynamics is
Oseledets multiplicative ergodic theorem,
(see \cite{Oseledets, BarreiraP, Barreira, Viana14} for precise statements),
which establishes the existence of Lyapunov exponents for
almost all points with respect to any borelian, probability and invariant measure. In the context of Oseledec's Theorem we call \emph{``regular''} a point for which are defined all Lyapunov exponents.
When a point is not regular we call it an \emph{``irregular''} point.

The importance of Oseledec's Theorem is that it provides a link between
ergodic properties of systems and geometric properties of the infinitesimal
displacements around orbits. This is, of course, the basis of
the very well known ``smooth ergodic theory'' or ``Pesin theory''
\cite{BarreiraP}.

In our results we use Lyapunov exponents to obtain a new description of rigidity of  Anosov diffeomorphisms on $\mathbb{T}^d, d \geq 3,$ analysing the set of regular points of the  Anosov diffeomorphisms.  Let us contextualize better.

\begin{definition}
Let $M$ be a $C^{\infty}$ compact, connected and boundaryless manifold and $f:M \rightarrow M$ be a diffeomorphism. We say that $f$ is an Anosov diffeomorphims if there are numbers $0 < \beta < 1 < \eta, C > 0$ and a continuous splitting  $T_xM = E^u_f(x) \oplus E^s_f(x),$  invariant over $Df,$  such that
$$ ||Df^n(x) \cdot v || \geq \frac{1}{C} \eta^n ||v||, \forall v \in E^u_f(x),$$
$$ ||Df^n(x) \cdot v || \leq C \beta^n ||v||, \forall v \in E^s_f(x).$$
\end{definition}

When $M = \mathbb{T}^d, d \geq 2,$ it is known by Franks, see \cite{FRANKS}, that an Anosov diffeomormphism $f: \mathbb{T}^d \rightarrow  \mathbb{T}^d $ is conjugated with its linearization $L:  \mathbb{T}^d \rightarrow  \mathbb{T}^d,$ that is, there is a continuous function $h: \mathbb{T}^d \rightarrow  \mathbb{T}^d $ such that $$ h \circ f = L \circ h,$$
where $L$ is given by the matrix of the linear isomorphism $f_{\ast} : \pi_1 (\mathbb{T}^d) =\mathbb{ Z}^d \rightarrow \mathbb{Z}^d.$

For a given $C^1-$Anosov diffeomorphism $f: M \rightarrow M$ denote by $R(f)$ the set of regular points of $f$ (in the Oseledec's Theorem sense). We recall that given $x \in R(f) $ and $ v \in T_xM \setminus \{0\},$

$$\displaystyle\lim_{|n| \rightarrow +\infty} \frac{1}{n} \log(|| Df^n(x) \cdot v ||) = \lambda (x, v).$$ Consider $x \in R(f),$ we can verify from definition that for a given $y \in R(f)$ such that $y \in W^{s,u}_f(x),$ then $\lambda(y,v) = \lambda(x,v).$
The value $\lambda (x,v) $ is called a Lyapunov exponent of $f$ in direction $v$ of $x.$ 

To find conditions for  $f$ and $L$ to be $C^{1}-$conjugated is in fact the core of rigidity questions. In this setting, Lyapunov exponents play an important role. The results in this direction are obtained by  comparison between Lyapunov exponents of $f$ and the  Lyapunov exponents of its linearization $L,$ as in \cite{Llave92, Go, GoGu} for instance.

Here we are proposing results of rigidity without making direct comparison between Lyapunov exponents of $f$ and those of its linearization $L.$ It is a different approach, if we compare with known results as presented in \cite{Llave92, Go, GoGu}. We prove that if $f$ has all Lyapunov exponents everywhere, then we get $C^1-$conjugacy with $L.$ Of course, in some moment we will need, under our assumptions, to prove coincidence of periodic data between $f$ and $L$ and apply old results.

The tooling novelty  here is to use  unstable entropies (see \cite{Hua}) to get coincidence of periodic data between $f$ and $L.$ Although in the end we can apply old results, we will present a different approach  to pass from continuous conjugacy to $C^1-$conjugacy by using conformal metrics. Our main results are.

\begin{maintheorem}\label{teo4}
Let $L: \mathbb{T}^d \rightarrow \mathbb{T}^d , d \geq 2,$ be a
linear Anosov automorphism, diagonalizable over $\mathbb{R},$
irreducible over $\mathbb{Q},$ such that its eigenvalues have distinct modulus. Let $f: \mathbb{T}^d \rightarrow \mathbb{T}^d , d \geq 3,$ be $C^{1+ \alpha}-$map with $ 0 < \alpha < 1,$ sufficiently $C^1-$close to $L$ (and hence Anosov).
If every point $x \in \mathbb{T}^d$ is (Lyapunov) regular for $f,$ then $f$ is
 $C^{1+\varepsilon}$ conjugated to $L,$ for some
 $\varepsilon > 0.$
\end{maintheorem}

In contraposition with the previous rigidity result, we
just recall that there are
some known examples (we formulate them as theorems to keep
the symmetry of the exposition) which lead to:

\begin{maintheorem}\label{teo7}
For any $d \ge 4$,
there are Anosov maps of $\torus^d$ which have
Lyapunov exponents everywhere and are not Lipschitz conjugate to
linear.   Such examples, can be found in any $C^\infty$
neighborhood  of linear automorphisms of $\torus^d$.
\end{maintheorem}

The examples above show that in higher dimensions, Lyapunov exponents
everywhere does not guarantee the $C^1$ conjugacy to linear, much
less the $C^\infty$.

%\section{Organization of the Proofs}
%
%For a more didactic presentation of the results, we first present preliminaries results concerning unstable entropies and after we use it to prove Theorem \ref{teo5}. In the sequence we establish a link between constant periodic data and the assumption every point is regular, so Theorem \ref{teo4} will follow as a consequence of Theorem \ref{teo5}.

\section{Preliminaries}

In \cite{Hua} the authors deal with a notion of topological entropy $h_{top}(f, \mathcal{W} )$ of an invariant expanding foliation $\mathcal{W}$ of a diffeomorphism $f. $ They establish a variational principle in this sense and a relation between $h_{top}(f, \mathcal{W} )$ and volume growth of $\mathcal{W}. $

Here $W(x)$  denotes the leaf of $\mathcal{W}$ by $x.$ Given  $\delta  > 0,$  we denote by $W(x, \delta)$ the $\delta-$ball centered in $x$ on $W(x),$ with the induced Riemannian distance, which is denoted by $d_{W}.$

Given $x \in M, $ $\varepsilon > 0, $ $\delta > 0$ and $n \geq 1$ an integer number, let $N_{W}(f, \varepsilon, n, x, \delta)$ be the maximal cardinality of all sets $S \subset \overline{W(x, \delta)}$ such that  $\displaystyle\max_{j =0,\ldots, n-1} d_{W}(f^j(a), f^j(b)) \geq \varepsilon,$ for any $a \neq b$ elements in $S.$

\begin{definition}\label{uentropy} The unstable entropy of $f$ on $M,$ with respect to the expanding foliation $\mathcal{W}$ is given by
$$h_{top}(f, \mathcal{W} ) = \lim_{\delta \rightarrow 0} \sup_{x \in M} h^{\mathcal{W}}_{top}(f, \overline{W(x, \delta)}), $$
where
$$h^{\mathcal{W}}_{top}(f, \overline{W(x, \delta)}) = \lim_{\varepsilon \rightarrow 0} \limsup_{n \rightarrow +\infty} \frac{1}{n} \log(N_{W}(f, \varepsilon, n, x, \delta)). $$
\end{definition}

Define $\mathcal{W}-$volume growth by
 $$\chi_{\mathcal{W}}(f) = \sup_{x \in M } \chi_{\mathcal{W}}(x, \delta), $$
where
$$ \chi_{\mathcal{W}}(x, \delta) = \limsup_{n\rightarrow +\infty} \frac{1}{n} \log(Vol(f^n(W(x, \delta)))).$$

Note that, since we are supposing $\mathcal{W}$ an expanding foliation, the above definition is independent of $\delta$ and the Riemannian metric.

\begin{theorem}[Theorem C and Corollary C.1 of \cite{Hua}]\label{teoH} With the above notations
$$h_{top}(f, \mathcal{W} ) = \chi_{\mathcal{W}}(f).$$
Moreover $h_{top}(f) \geq h_{top}(f, \mathcal{W}). $
\end{theorem}

From the hypothesis of Theorem \ref{teo4} we can suppose that the eigenvalues of $L$ satisfy $0< |\beta_1^s| < \ldots < |\beta_k^s| < 1 <  |\beta_1^u| < \ldots < |\beta_n^u|. $ The Lyapunov exponents of $L,$ are
$\lambda^s_i(L) = \log(|\beta_i^s|), i =1, \ldots, k$ and $\lambda^u_i(L) = \log(|\beta_i^u|), i =1, \ldots, n.$

Let us introduce a notation $E^{s,L}_{(1, i)} = E^s_1 \oplus \ldots \oplus E^s_i, i=1, \ldots, k$ and  $E^{u,L}_{(1, i)} = E^u_1 \oplus \ldots \oplus E^u_i, i=1, \ldots, n.$ If $j > i,$ we denote $E^{s,L}_{(i, j)} =  E^s_i \oplus \ldots \oplus E^s_j $ and $E^{u,L}_{(i, j)} =  E^u_i \oplus \ldots \oplus E^u_j. $

It is known by Pesin \cite{pesin2004lectures}, that if $f$ is $C^1-$close to $L,$ then $T\mathbb{T}^d$ admits a similar splitting
$E^s_f = E^{s,f}_1 \oplus E^{s,f}_2 \oplus \ldots \oplus E^{s,f}_k $ and  $E^u_L = E^{u,L}_1 \oplus E^{u,L}_2 \oplus \ldots \oplus E^{u,L}_n .$ As before, define $E^{u,f}_{(1,i)} = E^{u,f}_1 \oplus  \ldots \oplus E^{u,f}_i $ and $E^{s,f}_{(1,i)} = E^{s,f}_1 \oplus  \ldots \oplus E^{s,f}_i ,$ analogously, for $i < j,$ we define $E^{s,f}_{(i, j)}$  and $E^{u,f}_{(i, j)}.$

For $f$ we denote by $\lambda^u_{i}(x,f)$ the Lyapunov exponent of $f$ at $x$ in the direction $E^{u,f}_i, i = 1, \ldots, n$ and by  $\lambda^s_{i}(x,f)$ the Lyapunov exponent of $f$ at $x$ in the direction $E^{s,f}_i, i = 1, \ldots, k,$ in the cases that Lyapunov exponents are defined.

By continuity of each subbundle,  we can take the decomposition $E^s_f \oplus E^{u,f}_{(1,i)} \oplus E^{u,f}_{(i+1, n)}$  a uniform partially hyperbolic splitting.

Moreover, by \cite{B}, each $E^{u,f}_{(1,i)} = E^{u,f}_1 \oplus  \ldots \oplus E^{u,f}_i, $ is integrable to an invariant foliation $W^{u,f}_{(1, i)},$ with $i =1, \ldots, n.$ An analogous construction holds for stable directions. Denote by $W^{u,f}_i(x)$ the tangent leaf to $E^{u,f}_i(x),$ analogously we define the leaves $W^{s,f}_i(x).$ These leaves are defined by $W^{u,f}_i(x) = W^{u,f}_{(1, i)}(x) \cap W^{u,f}_{(i, n)}(x). $  By \cite{FPS}, since $f$ is $C^1-$close to $L,$ the conjugacy $h$ between $L$ and $f$ is such that $h(W^{u,L}_{(1,i)}(x)) = W^{u,f}_{(1,i)}(h(x)), i = 1, \ldots, n,$ the same holds for intermediate stable foliations.

Related with the assumption \emph{``every point is regular''} is the concept of constant periodic data. It is a more deeper condition that makes Theorem \ref{teo4} work.

\begin{definition}\label{cpd}
Let $f: M \rightarrow M$ be a local diffeomorphism. We say that $f$
has constant periodic data if for every periodic points $p$ and $q$ of $f,$
the matrixes $Df^{\tau}(p)$ and $Df^{\tau}(q)$ are conjugated, for any  integer $\tau$ such that
$f^{\tau}(p) = p$ and $f^{\tau}(q) = q.$ In particular the set of
Lyapunov exponents of $p$ and $q,$ are equal and each common
Lyapunov exponent has the same multiplicity for both.
\end{definition}

 There is link between \emph{``every point is regular''}  and constant periodic data.  In fact we can verify that  \emph{``every point is regular''}  and  constant periodic data are equivalent conditions for a $C^{1+ \alpha}-$Anosov diffeomorphism $f$ which is $C^1-$close to its linearization $L,$ a diagonalizable over $\mathbb{R}$ and irreducible over $\mathbb{Q}$ Anosov automorphism.

%\begin{theorem}\label{teo5}
%Let $L: \mathbb{T}^d \rightarrow \mathbb{T}^d , d \geq 3,$ be a
%linear Anosov automorphism, diagonalizable over $\mathbb{R},$
%irreducible over $\mathbb{Q}$ (i.e. the characteristic
%polynomial cannot be factored into two polynomials
%with rational coefficients with distinct eigenvalues which are furthermore
%real).
%
%Hence, we can write:
%\[
%  \begin{split}
%& E^s_L = E^s_1 \oplus E^s_2 \oplus \ldots \oplus E^s_k \\
%& E^u_L = E^u_1 \oplus E^u_2 \oplus \ldots \oplus E^u_n.
%\end{split}
%\]
%
%
%Let $f: \mathbb{T}^d \rightarrow \mathbb{T}^d , d \geq 3,$ be $C^{1+ \alpha}-$map with $ 0 < \alpha < 1,$ sufficiently $C^1-$close to $L$ (and hence Anosov).
%If $f$ has constant periodic data condition, then $f$ is
% $C^{1+\varepsilon}$ conjugated with $L,$ for some
% $\varepsilon > 0.$
%\end{theorem}

%
%Of course, since we first establish   \emph{``every point is regular''}  implies constant periodic data, so Theorem \ref{teo4} will follow as direct consequence of Theorem \ref{teo5}.

\begin{lemma}[Constant periodic data $\Rightarrow$ Lyapunov exponents everywhere]\label{propuniform} Let $L: \mathbb{T}^d \rightarrow \mathbb{T}^d$ be an Anosov linear automorphism, diagonalizable over $\mathbb{R}$ with distinct eigenvalues. If $f$ is a $C^{1+ \alpha}-$Anosov diffeomorphism sufficiently $C^1-$close to $L$ with constant periodic data, then every point $x \in \mathbb{T}^d$ is regular and $\lambda^{\ast}_i(x,f) = \lambda^{\ast}_i(p,f), \ast \in \{s,u\}$ and $p$ is any point in $Per(f).$ Moreover the limits taken as in Oseledec's Theorem converge uniformly.

\end{lemma}

\begin{proof}
We argue with Livsic's Theorem. Denote by $\Lambda^{u,f}_{1,i}$ the common value of the sum of the $i-$first unstable Lyapunov exponents of $f$ at periodic points, where $\lambda^u_1(x,f) < \ldots < \lambda^u_n(x,f),$ are the $n-$first unstable Lyapunov exponent of $f$ at a regular point $x.$ Denote by $Jac^u_{(1,i)}f(x)$ the jacobian of $Df(x): E^{u,f}_{(1,i)}(x) \rightarrow E^{u,f}_{(1,i)}(f(x)).$

We see that $\log(|Jac^u_{(1,i)}f(x)|) - \Lambda^{u,f}_{1,i}$ has zero average over every periodic
orbit.

Hence, by Livsic's theorem \cite{Livsic72,Bowen}, we can find a $C^{\varepsilon}$  function $\phi$, for some $\varepsilon > 0,$ such that
$\phi: \torus^d \rightarrow \real$ such that
\begin{equation}  \label{cohomology1}
\log(|Jac^u_{(1,i)}f(x)|) - \Lambda^{u,f}_{1,i} =   \phi(f(x))-  \phi(x).
\end{equation}

Equivalently
\begin{equation} \label{conformal1}
|Jac^u_{(1,i)}f(x)|  = e^{-\phi(x)} e^{\phi(f(x))}e^{\Lambda^{u,f}_{1,i}}.
\end{equation}

By induction \begin{equation} \label{conformal2}
|Jac^u_{(1,i)}f^n(x)|  = e^{-\phi(x)} e^{\phi(f^n(x))}e^{n\Lambda^{u,f}_{1,i}}.
\end{equation}

Since $\phi$ is continuous there is $C > 1,$ such that
$$ C^{-1}e^{n\Lambda^{u,f}_{1,i}}\leq |Jac^u_{(1,i)}f^n(x)| \leq C e^{n\Lambda^{u,f}_{1,i}},$$
so the convergence $$\frac{1}{n} \log(|Jac^u_{(1,i)}f^n(x)|) \rightarrow \Lambda^{u,f}_{1,i} $$
is uniform.

The same we can apply to each $E^{u,f}_i, E^{s,f}_i, E^{s,f}_{(1,i)}$ and their corresponding Lyapunov exponents and sums.

\end{proof}

As a consequence of Lemma \ref{propuniform} and  Theorem \ref{teoH}, we obtain the next Lemma.

\begin{lemma}\label{entropy} Let $f: \mathbb{T}^d \rightarrow \mathbb{T}^d, d \geq 3, $ be a $C^{1+ \alpha}-$Anosov diffeomorphism $C^1$ close to $L: \mathbb{T}^d \rightarrow \mathbb{T}^d, d \geq 3,$ where $L$ is as in Theorem \ref{teo4}. If $f$ has constant periodic data, then for the foliation $W^{u,f}_{(1, i)}$  tangent to $E^{u,f}_{(1, i)},$ holds the equality $h_{top}(f, W^{u,f}_{(1, i)}) = \displaystyle\sum_{j=1}^i \lambda^{u}_j(p, f),$ where  $p$ is any point in $Per(f).$
\end{lemma}

\begin{proof} Let $x \in \mathbb{T}^d  $ be an arbitrary point.
$$\lim_{n \rightarrow + \infty}\frac{1}{n} \log(Vol(f^n((W^{u,f}_{(1, i)}(x, \delta)))) = \lim_{n \rightarrow +\infty}\frac{1}{n} \log(Jac^u_{(1,i)} f^n(x)\cdot Vol(W^{u,f}_{(1, i)}(x, \delta)) ).$$
Using Lemma \ref{propuniform}, the right side of the above expression converges uniformly to $\sum_{j=1}^i \lambda^u_{j}(p,f)$ where $p $ is any point in $ Per(f).$ So for any $x \in \mathbb{T}^d,$ holds $\chi_{W^{u,f}_{(1, i)}} (x, \delta) = \sum_{j=1}^i \lambda^u_{j}(p,f).$ It implies $\chi_{W^{u,f}_{(1, i)}}(f) = \sum_{j=1}^i \lambda^u_{j}(p,f). $

By Theorem C of \cite{Hua}, we obtain  $h_{top}(f, W^{u,f}_{(1, i)}) = \displaystyle\sum_{j=1}^i \lambda^{u}_j(p,f),$ as required.
\end{proof}

Note that the Lemma \ref{propuniform} asserts that constant periodic data implies every point is regular. In the same setting we can prove the converse by using specification property.

\begin{lemma}[Lyapunov exponents everywhere $\Rightarrow$ Constant periodic data]\label{spc}
Let $f: \mathbb{T}^d \rightarrow \mathbb{T}^d , d \geq 3,$ be a $C^1$ Anosov diffeomorphism such that every point is regular and $f$ admits an invariant decomposition of the tangent bundle as sum of one dimension and $Df-$invariant sub bundles:
 \[
  \begin{split}
& E^s_f = E^s_1 \oplus E^s_2 \oplus \ldots \oplus E^s_k \\
& E^u_f = E^u_1 \oplus E^u_2 \oplus \ldots \oplus E^u_n,
\end{split}
\]
then $f$ has constant periodic data.
\end{lemma}

\begin{proof} The argument here is similar to Hopf argument, using local product structure. Denote by $J^u_if(x)$ the jacobian of $f$ restricted to $E^u_i$ at $x$ and $\lambda^u_i (x,f)$ the Lyapunov of $f$ at $x$ in the $E^u_i$ direction.

Let $x_0$ be an arbitrary point on $\mathbb{T}^d$ and consider the Lyapunov exponent $ \lambda^u_i (x_0,f).$  Since $f$ have local product structure, there is an open neighborhood $V$ of $x_0,$ such that,  given $z \in V,$ there is a point $z'\in V \cap W^u_f(z)\cap W^s_f(x_0).$ Since every point is regular we have

$$ \lambda^u_i (x_0,f) = \displaystyle\lim_{n \rightarrow +\infty}\frac{1}{n} \log(J^u_if^n(x_0)) =
\displaystyle\lim_{n \rightarrow +\infty}\frac{1}{n} \log(J^u_if^n(z')) = \displaystyle\lim_{n \rightarrow -\infty}\frac{1}{n} \log(J^u_if^n(z)) =  \lambda^u_i (z,f). $$

The map $x \mapsto \lambda^u_i (x,f)$ is locally constant. Since $\mathbb{T}^d$ is connect, $x \mapsto \lambda^u_i (x,f)$ is  constant on $\mathbb{T}^d.$

\end{proof}

\section{Proof of Theorem \ref{teo4}}

From Lemma \ref{propuniform} and Lemma \ref{spc}, in Theorem \ref{teo4} we can replace the condition \textit{every point is regular} by $f$\textit{ has constant periodic data} which we will use in the proof from now on.

\begin{proof}  Since $h(W^{u,L}_{(1, i)}) =  W^{u,f}_{(1, i)},$ it implies $h_{top}(f, W^{u,f}_{(1, i)}) = h_{top}(L, W^{u,L}_{(1, i)}).$ Now, consider $ \beta^s_i, i =1, \ldots, k,$ and  $ \beta^u_i, i =1, \ldots, n,$ the eigenvalues of $L,$ such that
$$0 < |\beta^s_1| <  |\beta^s_2| < \ldots < |\beta^s_k|< 1 < |\beta^u_1| <  |\beta^u_2| < \ldots < |\beta^s_n| .$$

Let $p$ be a periodic point of $f.$  Since $f$ has constant periodic data, so for any $i =1, \ldots, n$ by Lemma \ref{entropy}
$$\lambda^u_1(p, f) + \ldots + \lambda^u_i(p, f) =  h_{top}(f, W^u_{(1, i)}) = h_{top}(L, W^u_{(1, i)}(L) ) = \lambda^u_1(L) + \ldots + \lambda^u_i(L), $$
 for any $i =1, \ldots, n.$
So, for $i = 1,$
$$\lambda^u_1(p, f) = \lambda^u_1(L), $$
for $i = 2,$ we get $\lambda^u_1(p, f) + \lambda^u_2(p, f)  =  \lambda^u_1(L) + \lambda^u_2(L),$ since $\lambda^u_1(p, f) = \lambda^u_1(L),$ so
$$\lambda^u_2(p, f) = \lambda^u_2(L).$$
Analogously $\lambda^u_i(p, f) = \lambda^u_i(L), i =1, \ldots, n.$

Taking the  inverses, we obtain $$\lambda^s_i(p, f) = \lambda^s_i(L), i =1, \ldots, k,$$
note that $f$ and $L$ has the same periodic data, by \cite{Go} and \cite{SY}, the maps $f$ and $L$ are $C^{1+ \varepsilon}$ conjugated for some $\varepsilon > 0,$  if $f$ is enough  $C^1-$close to $L.$
\end{proof}

\begin{remark} From a remarkable result in \cite{GoBoot}, we observe that in dimension three, if $f$ is $C^{\infty}$ as in Theorem \ref{teo4}, then the conjugacy $h$ is also $C^{\infty}.$
\end{remark}

\section{From continuous to differentiable conjugacy}

In the previous section, we obtained that $f$ and $L$ have same periodic data. In \cite{Go, GoGu} the authors provide a proof of differentiability of the conjugacy by an argument involving Gibbs measures on intermediate foliations. This kind of argument is also applied in \cite{SY}. Here, we present a topological argument to pass from continuity to differentiability  the conjugacy in Theorem \ref{teo4}.

Let us introduce conformal distances on each invariant one dimensional leaf.

\begin{lemma} There exists a metric $d^u_i$ on each leaf $W^{u,f}_i(x)$ tangent to $E^{u,f}_i,$  such that  $d^u_i(f(a), f(b)) = e^{\lambda^u_i} d^u_i(a, b),$ where $\lambda^u_i$ the common value of the Lyapunov exponents of periodic points of $f$ and $L$ relative to directions $E^{u,i}_f$ and $E^{u,i}_L$ respectively.
\end{lemma}

\begin{proof}
Denote by $\lambda^u_i$ the common value of the Lyapunov exponents of periodic points of $f$ and $L$  in the directions $E^{u,f}_i$ and $E^{u,L}_i,$  respectively. Let us to denote on $\mathbb{T}^d,$ the $f-$invariant foliations $\mathcal{F}^{\ast,f}_i$ tangent to $E^{\ast,f}_i, \ast \in \{s,u\}.$

We see that $\log (||Df(x)|E^{u,f}_i(x)||) -\lambda^u_i$ has zero average over every periodic
orbit.

Since $f$ is a $C^{1+ \alpha}-$Anosov diffeomorphism, the map $x \mapsto \log (||Df(x)|E^{u,f}_i(x)||) $  is uniform $C^{\varepsilon}$ on $\mathbb{T}^d,$ for some $\varepsilon > 0.$ Hence, by Livsic's theorem \cite{Livsic72,Bo74}, we can find a $C^{\varepsilon}-$function $\phi^u_i$ such that,
$\phi^u_i: \mathbb{T}^d \rightarrow \mathbb{R}$ such that
\begin{equation}  \label{cohomology}
\log (||Df(x)|E^{u,f}_i(x)||) - \lambda^u_i =    \phi^u_i(f(x)) - \phi^u_i(x).
\end{equation}

Equivalently
\begin{equation} \label{conformal}
e^{\phi^u_i(x)}||Df(x)|E^{u,f}_i(x)|| e^{-\phi^u_i(f(x))} = e^{\lambda^u_i}.
\end{equation}

We can interpret \eqref{conformal} as saying that, if we define a
metric, conformal to the standard metric in the torus by a factor
$e^{-\phi^u_i},$ then for a convenient metric $f$ expands on $W^{u,f}_i-$leaves  by exactly
$e^{\lambda^u_i}.$

In fact, fix an orientation on $W^{u,f}_i(x)$ and consider $a \geq b$ on $W^{u,f}_i(x),$ consider the metric
$$d^u_i(a, b) = \int_a^b e^{-\phi^u_i(x)}dx, $$ where $dx$ denotes the infinitesimal size on $W^{u,f}_i(x).$ With this

$$ d^u_i(f(a), f(b)) = \int_{f(a)}^{f(b)} e^{-\phi^u_i(y)}dy = \int_a^b e^{-\phi^u_i(f(x))}||Df(x)|E^{u,f}_i(x)|| dx =$$ $$= e^{\lambda^u_i} \int_a^b e^{-\phi^u_i(x)}dx = e^{\lambda^u_i} d^u_i(a,b).$$

\end{proof}

Also we need the following proposition.

\begin{proposition}[Proposition 8.2.2 of \cite{AH}]\label{propAH} Let $L : \mathbb{R}^n \rightarrow \mathbb{R}^n$ be a hyperbolic linear automorphism
and let $T : \mathbb{R}^n \rightarrow \mathbb{R}^n$ be a homeomorphism. If $\bar{d}( L, T)$ is finite, then there is a
unique map $\phi : \mathbb{R}^n \rightarrow \mathbb{R}^n$ such that
\begin{enumerate}
\item $L \circ \phi = \phi \circ T,$
\item $\bar{d}(\phi, id_{\mathbb{R}^n})$ is finite.

\hspace{-1.7cm}Furthermore, for $K > 0$ there is a constant $\delta_K > 0$ such that if $\bar{d}(L,T) < K,$

\hspace{-1.7cm}then the above map $\phi$ has the following properties :

\item $\bar{d}(\phi, id_{\mathbb{R}^n}) < \delta_K,$
\item  $\phi$ is a continuous surjection,
\item  $\phi$ is uniformly continuous under $\bar{d}$ if so is T.
\end{enumerate}
\end{proposition}

To pass from continuity to differentiability we will make an induction process based on Gogolev method \cite{Go}. In this work, it is proved the following induction steps:

\begin{enumerate}
\item If $h$ is $C^{1 + \nu}$ on $W^{u,f}_{1, m-1}$ and $h(W^{u,f}_i) = W^{u,L}_i, i = 1, \ldots, m-1,$ then $h(W^{u,f}_m) = W^{u,L}_m.$
\item If $h$ is $h(W^{u,f}_m) = W^{u,L}_m, m = 1, \ldots, n,$ then $h$ is $C^{1+ \alpha}$ restricted on each $W^{u,f}_m.$
\end{enumerate}

The proof of the step $(1)$ is topological and the one of step $(2)$ is based on a construction of a Gibbs measure on each leaf $W^{u,f}_m.$ Assuming the topological argument in the step $(1)$ we prove step $(2)$ via conformal metrics.

As we said before, by \cite{FPS}, since $f$ is $C^1-$close to $L,$ the conjugacy $h$ between $L$ and $f$ is such that $h(W^{u,L}_{(1,i)}(x)) = W^{u,f}_{(1,i)}(h(x)), i = 1, \ldots, n,$ the same holds for stable foliations. Assuming topological step (1) from Gogolev in \cite{Go}, to pass to continuity from differentiable is sufficient to prove the next Lemma and finalize the proof using Journ\'{e}'s Lemma, as we will see latter.

\begin{lemma} Suppose that $h$ is $h(W^{u,L}_m) = W^{u,f}_m, m = 1, \ldots, n,$ then $h$ is $C^{1+ \varepsilon}$ restricted on each $W^{u,f}_m, m = 1, \ldots, n,$ for some $\varepsilon > 0$ enough small.
\end{lemma}

\begin{proof}

We go to prove the differentiability of the conjugacy between $f$ and $L,$ by using the conformal metrics on each one dimensional invariant foliation of $f.$

Let $h: \mathbb{T}^d \rightarrow \mathbb{T}^d$ be the conjugacy between $f$ and $L,$ such that $$h \circ L = f \circ h.$$

We first observe that, since $h$ sends $W^{u,f}_i$ leaves in $W^{u,L}_i$ leaves then $h$ induces naturally a conjugacy $\mathcal{H}: \mathbb{T}^d/\mathcal{F}^{u,L}_i \rightarrow \mathbb{T}^d/\mathcal{F}^{u,f}_i.$

Let us introduce a leaf equivalence on the unstable leaves $W^{u,f}_i.$ We say that two unstable leaves $W$ and $W',$ tangent to $E^{u,L}_i,$ are related if there is an integer $n$ such that $f^n(W) = W'.$
For each equivalence $[W]$ class choose a representantive $W, $ and a point  $a_0 \in W.$ Fix orientations on the foliations $\mathcal{F}^{u,L}_i$ and $\mathcal{F}^{u,L}_i$ and suppose that $h$ preserves the fixed orientation.  Up to change $(L, f)$ by $(L^2, f^2)$ we can suppose  $L$ and $f$ preserve the orientations established.

 Using this orientation, choose points $a_j , j \in \mathbb{Z}$ such that $a_{j} < a_{j+1}$ and $|a_j - a_{j+1}| = 1,$ where $|u - v|$ is the euclidean distance iduced on $W.$ In fact we are seeing $W$ as a real line. Let $b_j = h(a_j), j \in \mathbb{Z}.$ For each $j$ we choose a function $\phi^u_{i_j}$ such that $d^u_i$ is such that $d^u_i (b_j, b_{j+1}) = 1.$ To simplify the writing,  we denote by $[p, q]$ a segment connecting points $p$ and $q$ on a leaf of type $W^{u,L}_i$ and $W^{u,f}_i.$ The same notation we will use for leaves lifted on $\mathbb{R}^d.$

Let us to define a map $\tilde{h}: [a_j, a_{j+1}] \rightarrow [b_j, b_{j+1}], $ using $\phi^u_{i_j}$ and the corresponding $d^u_i$ such that $\tilde{h}(\theta)$  is the unique point $p$ in $[b_j, b_{j+1}]$ such that $d^u_i(b_j , p ) = |a_j - \theta|.$  Also, for the given $j,$ using $\phi^u_{i_j}$ and the corresponding $d^u_i$ we define $\tilde{h}: [L^n (a_j), L^n(a_{j+1})] \rightarrow [f^n(b_j), f^n(b_{j+1})] $ following the same strategy before, for each $n \in \mathbb{Z}.$ By construction, $\tilde{h}$ and $h$ coincide on the extremes of intervals, as defined.

We have defined a map $\tilde{h}$ on every leaf of $[W],$ moreover it satisfies $\tilde{h} \circ L = f \circ \tilde{h}.$ In fact, consider $\theta \in [a_0, a_1]$ such that $|a - \theta | = \alpha.$ By definition $d^u_i (\tilde{h}(a_0), \tilde{h}(\theta)) = \alpha.$ Taking the first iterated we get $|L(a_0) - L(\theta)| = e^{\lambda^u_i}\alpha$ and $d^u_i(f(\tilde{h}(a_0)) , f(\tilde{h}(\theta))) = e^{\lambda^u_i} d^u_i (\tilde{h}(a_0), \tilde{h}(\theta)) = e^{\lambda^u_i} \alpha.$ By definition $f (\tilde{h} (\theta)) = \tilde{h}(L(\theta)).$ The same works for any $[a_j, a_{j+1}]$ and its iterated by $L^n, n \in \mathbb{\Z}.$  Varying on all equivalence classes we get a new map $\tilde{h}: \mathbb{T}^d \rightarrow \mathbb{T}^d , $ such that $\tilde{h} \circ L = f \circ \tilde{h}.$ Since $\tilde{h}$ is bijective restricted to each leaf $W^{u,L}_i$ and $\tilde{h}(W^{u,L}_i) = h(W^{u,L}_i) ,$ so $\tilde{h}$ is a bijection.

We can describe $\tilde{h}$ as a solution of a specific ordinary differential equation. In fact, given a leaf $W = W^{u,L}_i,$  $\tilde{h}:[a_0, a_1] \rightarrow [b_0, b_1]$ is defined by
\begin{equation}\label{ODE}
z' = e^{\phi^u_{i_0}(z)}, z(a_0) = b_0.
\end{equation}

In fact, let $z:[a_0,a_1] \rightarrow [b_0, b_1]$ be a solution of the differential equation $(\ref{ODE}).$ Let $a_0 \leq \theta \leq a_1,$ we have
$z'(t)e^{-\phi^u_{i_0}(z(t))} = 1,  $ for any $t \in [a_0, a_1],$ so

$$ \theta - a_0 = \int_{a_0}^{\theta} e^{-\phi^u_{i_0}(z(t))} z'(t) dt =  \int_{z(a_0)}^{z(\theta)} e^{-\phi^u_{i_0}(s)} ds  = d^u_i(z(a_0),z(\theta) ) = d^u_i(b_0,z(\theta) ) ,   $$
here $ds$ denote  the infinitesimal length arc of $W^{u,f}_i(b_0),$  so $z(\theta) = \tilde{h}(\theta).$  The same can be done for any values $j$ and intervals $[L^n(a_j), L^n(a_{j+1})], n \in \mathbb{Z},$ on $W^{u,L}_i-$leaves. In particular the differential equations of kind $(\ref{ODE})$ have unique solution.

Let $H: \mathbb{R}^d \rightarrow \mathbb{R}^d $ be the lift of $h,$ and $\pi: \mathbb{R}^d \rightarrow \mathbb{T}^d,$ the natural projection. In $\mathbb{R}^d,$ consider fundamental domains $D$ of kind $[0,1)^d + c,$ for $c \in \mathbb{Z}^d.$ For a domain $D,$ we define $\tilde{H}: D \rightarrow H(D), $ given by
$$\tilde{H}(q) = (\pi_{|H(D)})^{-1}(\tilde{h}(\pi(q))),  $$
in other words $\tilde{H}(q)$ is the unique point in $ p \in H(D)$ such that $\pi(p) = \tilde{H}(\pi(q)).$ Particularly $\tilde{H}(q + c) = \tilde{H}(q),$ for any $q \in \mathbb{R}^d$  and $c \in \mathbb{Z}^d.$  Since $\tilde{h}$ is bijection, $\tilde{H}$ so is.

Let $\bar{f}, \bar{L}: \R^d \rightarrow \R^d $ be the lifts of $f$ and $L$  respectively we have $\tilde{H}\circ \bar{L} = \bar{f}\circ \tilde{H}. $

Consider $W$ a  $ W^{u,L}_i-$leaf for which we have chosen points $a_j, j \in \mathbb{Z}$ and $\overline{W}$ a lift of $W$ in $\mathbb{R}^d.$ Let $\bar{a}_0, \bar{a}_1$ points  in $\overline{W}$ such that $\pi(\bar{a}_j) = a_j, j=0,1.$ Suppose that the segment $[\bar{a}_0, \bar{a}_1]$ is  contained in  $\overline{W},$ and connecting $\bar{a}_0, \bar{a}_1,$ crosses domains $D_1,  \ldots , D_k.$ Consider $\gamma_1, \ldots, \gamma_k$ such that $\gamma_i = D_i \cap [\bar{a}_0, \bar{a}_1], $ the connect component of $\overline{W} \cap D_i.$ Let $\delta_i = H(\gamma_i) \cap H(D_i),$  the connect component of $H([\bar{a}_0, \bar{a}_1]) \cap H(D_i).$

Since $\tilde{h}$  is constructed as solution of an O.D.E, $\tilde{h}(a_i) = h(a_i), i=0,1$ and $h([a_0, a_1]) = \tilde{h} ([a_0, a_1]),$ we get $\tilde{H}:[\bar{a}_0, \bar{a}_1] \rightarrow H([\bar{a}_0, \bar{a}_1])$ is a homeomorphism. By continuity $\tilde{H}(\gamma_i) \subset \delta_i,$ since $\tilde{H}([\bar{a}_0, \bar{a}_1])$ is connected, then  $\tilde{H}(\gamma_i)$ and $\tilde{H}(\gamma_{i+1})$ are connected by extremes. We conclude $\tilde{H}(\gamma_i) = \delta_i = H(\gamma_i).$ Arguing similarly using segments $[\overline{L^n(a_j)}, \overline{L^n(a_{j+1})}], j, n \in \mathbb{Z},$ for all $W^{u,L}_i-$leaves, we conclude that

\begin{equation} \label{distance1}
x \in D \Rightarrow H(x), \tilde{H}(x) \in H(D).
\end{equation}

So there is $K > 0$ such that

\begin{equation} \label{distance2}
x \in \mathbb{R}^d \Rightarrow || H(x) - \tilde{H}(x)|| \leq K .
\end{equation}

Finally, since $H$ is the lift of $h,$ we have $|| H(x) - x|| \leq R,$ for any $x \in \mathbb{R}^d$ and we conclude

\begin{equation} \label{distance3}
x \in \mathbb{R}^d \Rightarrow ||  \tilde{H}(x) - x|| \leq R + K .
\end{equation}

By Proposition $\ref{propAH}$ we conclude $H = \tilde{H},$ consequently $h = \tilde{h}.$ Note that by $(\ref{ODE})$ the conjugacy $h$ restricted to $W^{u,L}_i$ leaves are $C^{1+\varepsilon},$ for some small $\varepsilon > 0,$ except possibly at points of type $L^n(a_j).$  We observe that we can do the same construction with points $c_j$ on leaves $W,$ such that $c_j$ is the middle point between $[a_j, a_{j+1}].$ We so conclude  the conjugacy $h$ restricted to $W^{u,L}_i$ leaves are $C^{1+\varepsilon},$ for some small $\varepsilon > 0,$ except possibly at points of type $L^n(c_j).$ Since the sets of points $L^n(a_j)$ and points $L^n(c_j)$ are mutually disjoint, we conclude that $h$ restricted to $W^{u,L}_i$ leaves is in fact $C^{1+\varepsilon},$ for some small $\varepsilon > 0.$

\end{proof}

To finalize our argument we evoke Journ\'{e}'s Lemma.

\begin{lemma}[ Journ\'{e}'s Lemma, \cite{Journe}] Let $W$ and $V$ be two mutually transverse uniformly
continuous foliations with $C^r$
leaves on a manifold $M.$ Let $\varphi: M \rightarrow \mathbb{R}$ be a function.
Assume that $\varphi \in C^{r + \nu}_V(M) \cap C^{r + \nu}_W(M), $ for some $\nu \geq 0.$   Then $\varphi$ is $C^{r + \nu},$ if $\nu > 0,$ otherwise $\varphi$ is $C^{r-\varepsilon},$ for any $\varepsilon > 0.$
\end{lemma}

Note that leaves of type $W^{u,L}_1$ and $W^{u,L}_2$ are transversal and subfoliate $W^{u,L}_{1,2},$ so by Journ\'{e}'s Lemma \cite{Journe}, we get $h$ is  uniformly $C^{1 + \varepsilon},$ for some $\varepsilon > 0,$ enough small, on the unstable leaves $W^{u,L}_{(1,2)}.$ Inductively $h$ is uniformly $C^{1 + \varepsilon},$ for some $\varepsilon > 0,$ enough small, on the unstable leaves $W^{u,L}_{(1,n)} = W^u_L.$ Analogously $h$ is  uniformly $C^{1 + \varepsilon},$ for some $\varepsilon > 0,$ enough small, on the stable leaves $W^s_L.$ Finally, by Journ\'{e}'s Lemma \cite{Journe}, $h$ is $C^{1 + \varepsilon},$ for some $\varepsilon > 0.$

\section{A brief comment on the previous proof}

In the previous proof, for each equivalence class $[W]$ of leaves of a type of $L-$invariant leafs we find functions $\phi_j$ defined on $[L^n(a_j), L^n(a_{j+1})], n \in \mathbb{Z}.$  The equations of kind $(\ref{ODE})$ determined the conjugagy. In fact it is possible to prove that there is a unique choice of $\phi_j$ for all $\mathbb{T}^d.$ We go to explain below.
Suppose that for $[a_0, a_1] \subset W$ we have determined the function $\phi.$ Consider $W$ an expanding leaf of an invariant foliation $W^{u,f}_i$ such that it has dense orbit. Moreover, for any $\varepsilon > 0$ there is $N \geq 0,$ integer such that, $W_n := L^n([a_0, a_1])$ is $\varepsilon-$dense on $\mathbb{T}^d$ for any $n \geq N.$ Suppose $I$ a closed interval on a leaf of $W^{u,f}_i.$ By the previous section, the conjugacy $h$ restricted to $I$ is defined by an equation $$z' = e^{\psi (z)}, z(x_0) = y_0.$$

On $W_n$ consider intervals $J_n \subset W_n$ and points $x_n \in J_n$ such that $J_n \rightarrow I$ and $x_n \rightarrow x.$ Consider $z_n,$ the conjugacy $h$ restricted to $J_n,$ is defined by $$z' = e^{\phi(z)}, z(x_n) = y_n.$$ By uniform continuity of $h,$ we conclude $z_n \rightarrow z$ uniformly. By the other hand $z_n \rightarrow u$ a solution of $u' = e^{\phi (u)}, u(x_0) = y_0.$ So $z = u$ in $I.$ 

Finally $e^{\phi(z(t))} - e^{\psi(z(t))} = 0,$ for any $t \in I,$ we conclude that $e^{\phi(x)} = e^{\psi(x)},$ for any $x \in h(I).$ Then $\phi = \psi.$  Do it for all $I$ and so $\phi = \psi.$

\section{Comments on dimensions two and three}

In dimensions two and three, using the same techniques to prove Theorem \ref{teo4} we can obtain more stronger versions of rigidity results.

\subsection{Dimension two} In the case of $\mathbb{T}^2$ consider $f: \mathbb{T}^2 \rightarrow \mathbb{T}^2 $ a $C^{1+\alpha}-$Anosov diffeomorphism and  $L: \mathbb{T}^2 \rightarrow \mathbb{T}^2$ its linearization of $f.$ Consider $h$ the conjugacy between $f$ and $L.$ If the diffeomorphism $f$ has all Lyapunov exponents defined everywhere, using Lemma \ref{spc} we obtain $f$ has constant periodic data. Since $h$ preserves invariant foliations, using  Lemma \ref{entropy} we conclude that $f$ and $L$ have same periodic data. By \cite{Llave92}, $f$ and $L$ are $C^{1 + \varepsilon}$ conjugated. Moreover if $f$ is $C^{\infty},$ then $h $ so is.

\subsection{Dimension three} In the case of $\mathbb{T}^3$ consider $f: \mathbb{T}^3 \rightarrow \mathbb{T}^3 $ a $C^{1+\alpha}-$Anosov diffeomorphism such that $E^u_f = E^{wu}_f \oplus E^{su}_f$ respectively weak and strong unstable directions which are invariant by $Df.$ If $L: \mathbb{T}^3 \rightarrow \mathbb{T}^3$ is the linearization of $f,$ by \cite{H} is known that $E^u_L = E^{wu}_L \oplus E^{su}_L.$ Considering $h$ the conjugacy between $f$ and $L,$ since $f$ has all Lyapunov exponents defined everywhere, using  Lemma \ref{spc} we obtain $f$ has constant periodic data and since $h$ applies leaves $W^{wu}_f$ in $W^{wu}_L,$ as in Lemma \ref{entropy} we conclude that $f$ and $L$ have same periodic data. By \cite{GoGu}, $f$ and $L$ are $C^{1 + \varepsilon}$ conjugated.

\section{Proof of Theorem \ref{teo7}}

Here we give an outline of the construction of examples in \cite{Llave92}. In fact such construction provide counter-examples of Theorem \ref{teo4} in the absence of the irreducibility over $\mathbb{Q}$ hypothesis.

In \cite{Llave92} the author describes how to obtain a $C^{\infty}-$Anosov diffeomorphisms  $f: \mathbb{T}^d \rightarrow \mathbb{T}^d, d \geq 4,$ arbitrarily $C^1-$close to a linear Anosov automorphism, which is $C^k$ but not $C^{k+1}$ conjugated to its corresponding linearization.
In few lines, let $A: \mathbb{T}^2 \rightarrow \mathbb{T}^2,$ and $B:  \mathbb{T}^{d-2} \rightarrow \mathbb{T}^{d-2} $ be a linear Anosov automorphisms with simple real spectrum. Take $1 \leq n < m  $ integer numbers and consider $f(x,y) = (A^nx,B^my + \psi(x) e_u ),$ where $\psi: \mathbb{T}^2 \rightarrow \mathbb{R} $ is a enough small $C^{\infty}-$map and $e_u$ is an unstable eigenvector of $B,$ with $Be_u = \lambda \cdot e_u$.  Since $\psi $ is enough small then $f$ is a $C^{\infty}-$Anosov map $C^1-$close to $L(x,y) = (A^nx,B^my). $ The numbers $m,n$ can be chosen such that $L$ is diagonalizable over $\mathbb{R}$ with distinct eigenvalues. Let $\mu,$ such that $|\mu| > 1$ be the unstable eigenvalue of $A.$  By \cite{Llave92}, it is possible to choose $\psi $ sufficiently small such $f$ and $L$ are $C^{\alpha}$ conjugated to $L$ for any $0 \leq \alpha < \frac{n \log(|\mu|)}{m \log(|\lambda|)},$ but not $C^{\alpha}$ conjugated, for any $\alpha > \frac{n \log(|\mu|)}{m \log(|\lambda|)}.$ Since we arrange correctly $n,m$ we conclude that $f$ and $L$ are $C^0$ conjugated but not Lipschitz conjugated.

On the other hand, by construction, the derivative of $f$ is given by
\begin{equation}\label{block}
Df(x,y) = \left[
\begin{array}{ccc}
A^n &\theta(x,y)\\
0 &B^m
\end{array}
\right].
\end{equation}

The equation \eqref{block} implies that $f$ has same constant periodic data, since product of matrixes of kind given by \eqref{block} is a matrix with this same type. For $L$ we have $E^u_L = E^{u, L}_1 \oplus \ldots  \oplus E^{u, L}_r$ and $E^s_L = E^{s,L}_1 \oplus \ldots  \oplus E^{s,L}_k,$ all subbundles with dimension one. By \cite{pesin2004lectures}, for $f$ we get $E^u_f = E^{u, f}_1 \oplus \ldots  \oplus E^{u, f}_r$ and $E^s_f = E^{s,f}_1 \oplus \ldots  \oplus E^{s,f}_k,$ since $\psi$ can be taken sufficiently small.
Since $f$ has constant periodic data, with same periodic data of $L,$ by applying Lemma \ref{propuniform}, every point is regular. But $f$ and $L$ are not $C^1-$ conjugated to $L.$

The difference of the above example and Theorem \ref{teo4} is the fact that the characteristic polynomial $P$ of $Df^{\tau}(p),$ is such that $$P = P_1\cdot P_2,$$
where $p $ is a periodic point of $f,$ with period $\tau > 0$ and $P_1, P_2$ are the characteristic polynomial of $A^{\tau n}$ and $B^{\tau m}$ respectively. So $P$ doesn't satisfy the assumption of irreducibility over $\mathbb{Q}.$

In fact, given $r \geq 0$ an integer number,  by density of $\mathbb{Q},$ by a suitable choice of the integers $m,n$ as above,  it is possible to obtain $f: \mathbb{T}^d \rightarrow \mathbb{T}^d ,$ a $C^{\infty}-$Anosov diffeomorphism such that every point is regular such that  $f$ is $C^r$ but not $C^{r+1}$ conjugated to it is linearization $L.$

\end{document}